\documentclass[preprint,12pt]{elsarticle}%
\usepackage{amsfonts}
\usepackage{amsmath}
\usepackage{amssymb}
\usepackage{graphicx}%
\providecommand{\U}[1]{\protect\rule{.1in}{.1in}}
\journal{Indag. Math.}
\newtheorem{theorem}{Theorem}

\newtheorem{corollary}[theorem]{Corollary}

\newtheorem{lemma}[theorem]{Lemma}

\newenvironment{proof}[1][Proof]{\noindent\textbf{#1.} }{\ \rule{0.5em}{0.5em}}
\begin{document}
%
\begin{frontmatter}%
%

\title{Unit preserving linear functionals on clean vector lattices}%
%

\author{Fethi Benamor}%
%

\address{Research Laboratory LATAO
Department of Mathematics Faculty of Sciences of Tunis
Tunis-El Manar University 2092 Tunis, Tunisia
fethi.benamor@ipest.rnu.tn}%
%

\begin{abstract}%

Let $L$ be a vector lattice with a strong unit $e>0$.$\mathfrak{\ }$We show
that a unital linear functional $H:L\rightarrow\mathbb{R}$ satisfies $H\left(
u\right)  \neq0$ for any strong unit $u\in L$ if and only if $H$ acts like a
lattice homomorphism on every clean vector subspace of $L.$ We deduce that $L$
is clean if and only if any unital linear functional $H:L\rightarrow
\mathbb{R}$ such that $H\left(  u\right)  \neq0$ for any strong unit $u\in L$
is a lattice homomorphism.%

\end{abstract}%
%

%

\end{frontmatter}%

\section{Introduction}

The classical Gleason-Kahane-Zelazko theorem \cite{Zelazko} asserts that, for
a linear functional $H:A\rightarrow\mathbb{C}$ where $A$ is complex Banach
algebra, $H$ is a non-zero algebra homomorphism if and only if $H\left(
u\right)  \neq0$ for any invertible element $u\in A.$ In \cite{AzBou}, Azouzi
and Boulabiar gave a lattice version of this result. Precisely, they showed
the following. Let $L$ be a vector lattice with a strong unit $e>0$ and denote
by $L_{\mathbb{C}}$ its complexification. For a linear functional
$H:L_{\mathbb{C}}\rightarrow\mathbb{C}$ such that $H\left(  e\right)  >0$, $H$
is complex lattice homomorphism if and only if $H\left(  u\right)  \neq0$ for
all strong units $u\in L_{\mathbb{C}}.$ Unfortunately, a real version of the
latter result cannot be expected. Indeed, the linear functional $H:C\left(
\left[  0,1\right]  \right)  \rightarrow\mathbb{R}$, $f\mapsto\int_{0}^{1}f$
is such that $H\left(  f\right)  \neq0$ for all strong units $f\in C\left(
\left[  0,1\right]  \right)  $ but $H$ is not a lattice homomorphism. In this
work we intend to characterize linear functionals $H:L\rightarrow\mathbb{R}$
such that $H\left(  e\right)  =1$ and $H\left(  u\right)  \neq0$ for all
strong units $u\in L.$ This leads to a new characterization of the cleanness
of a vector lattice with a strong unit.

\section{Linear functionals preserving units}

We assume that the reader is familiar with the notion of vector lattices (also
called Riesz spaces). In this regard, we refer to the monograph
\cite{Aliprantis} for unexplained terminology and notation.

Let $L$ be a vector lattice. An element $u\in L$ is said to be a
\textit{strong unit} in $L$ whenever for all $x\in L$ we can find $\alpha>0$
such that $\left\vert x\right\vert \leq\alpha\left\vert u\right\vert .$ It is
clear that $u$ is a strong unit in $L$ if and only if $\left\vert u\right\vert
$ is a strong unit in $L.$

\textbf{From now on we assume that }$L$\textbf{\ is a vector lattice with a
distinguished strong unit }$e>0.$

By a \textit{component }of $L$ we mean an element $p\in L$ such that
$p\wedge\left(  e-p\right)  =0.$ Observe that $p$ is a component of $L$ if and
only if $e-p$ is a component of $L.$ Following \cite{Hager}, we say that $x\in
L$ is \textit{clean} if we can find a component $p$ of $L$ and a strong unit
$u$ in $L$ such that $x=p+u$. Any strong unit is obviously clean but the
converse need not be true. A subset $E$ of $L$ is said to be \textit{clean}
whenever all elements of $E$ are clean. Let $H:L\rightarrow\mathbb{R} $ be a
linear functional. We will say that $H$ is \textit{unital} if $H\left(
e\right)  =1$, and \textit{unit preserving} whenever $H\left(  u\right)
\neq0$ for all strong units $u\in L.$ Also $H$ will be called
a\textit{\ lattice homomorphism} if $\left\vert H\left(  x\right)  \right\vert
=H\left(  \left\vert x\right\vert \right)  $ for all $x\in L.$ If $E$ is a
subset of $L$, we will say that $H$ acts on $E$ like a lattice homomorphism if
$\left\vert H\left(  x\right)  \right\vert =H\left(  \left\vert x\right\vert
\right)  $ for all $x\in E.$

The following lemma is necessary to prove the main result in this work.

\begin{lemma}
\label{1}Let $H:L\rightarrow\mathbb{R}$ be a unital linear functional$.$

\begin{enumerate}
\item If $H$ is unit preserving then

\begin{enumerate}
\item $H$ is positive,

\item $H\left(  p\right)  \in\left\{  0,1\right\}  $ for any component $p\in
L$, and

\item $H\left(  p\right)  H\left(  q\right)  =0$ for every components $p,q\in
L$ with $p\wedge q=0.$
\end{enumerate}

\item If $H$ acts like a lattice homomorphism on the set of all strong units
of $L$ then

\begin{enumerate}
\item $H$ is positive, and

\item $H\left(  p\right)  \in\left\{  0,1\right\}  $ for any component $p\in
L$.
\end{enumerate}
\end{enumerate}
\end{lemma}

\begin{proof}
$(1)$ Assume that $H\left(  u\right)  \neq0$ for any strong unit $u.$

$(a)$ Let $x\in L^{+}$ and observe that $x+\lambda e$ is a strong unit for all
$\lambda>0.$ So $H\left(  x\right)  +\lambda=H\left(  x+\lambda e\right)
\neq0$ for all $\lambda>0,$ from which we derive that $H\left(  x\right)
\geq0.$

$(b)$ Let $p\in L$ be a component and observe that $p-\alpha e$ is a strong
unit for all real $\alpha\notin\left\{  0,1\right\}  .$ So $H\left(  p\right)
-\alpha=H\left(  p-\alpha e\right)  \neq0$ for all real $\alpha\notin\left\{
0,1\right\}  .$ Then $H\left(  p\right)  \in\left\{  0,1\right\}  .$

$(c)$ Observe first that $H\left(  p\right)  ,H\left(  q\right)  \in\left\{
0,1\right\}  .$ Also $p+q\leq e$ from which we derive that $H\left(  p\right)
+H\left(  q\right)  =H\left(  p+q\right)  \leq H\left(  e\right)  =1.$ So
$H\left(  p\right)  H\left(  q\right)  =0.$

$(2)$ Assume that $H$ acts as a lattice homomorphism on the set of strong
units of $L.$

$(a)$ Let $x\in L^{+}$ and observe that $x+\lambda e$ is a strong unit for all
$\lambda>0.$ So
\[
\left\vert H\left(  x\right)  +\lambda\right\vert =\left\vert H\left(
x+\lambda e\right)  \right\vert =H\left(  \left\vert x+\lambda e\right\vert
\right)  =H\left(  x+\lambda e\right)  =H\left(  x\right)  +\lambda\text{ for
all }\lambda>0.
\]
It follows that $\left\vert H\left(  x\right)  \right\vert =H\left(  x\right)
$ and then $H\left(  x\right)  \geq0.$

$(b)$ Observe that $g=2p-e$ is a strong unit and that $\left\vert g\right\vert
=e$. So $\left\vert 2H\left(  p\right)  -1\right\vert =\left\vert H\left(
g\right)  \right\vert =H\left(  \left\vert g\right\vert \right)  =H\left(
e\right)  =1.$Thus $H\left(  p\right)  \in\left\{  0,1\right\}  .$
\end{proof}

In the next theorem we give a characterization of unit preserving unital
linear functionals using their behavior on some remarkable subsets of $L.$

\begin{theorem}
Let $H:L\rightarrow\mathbb{R}$ be a unital linear functional. Then the
following are equivalent.

\begin{enumerate}
\item $H$ is unit preserving.

\item $H$ acts as a lattice homomorphism on the set of all strong units of $L
$.

\item $H$ acts as a lattice homomorphism on any clean vector subspace of $L.$
\end{enumerate}
\end{theorem}

\begin{proof}
$(1)\Rightarrow(2)$ Let $u\in L$ be a strong unit and pick $\alpha>0$ such
that $e\leq\alpha\left\vert u\right\vert .$ Now put $p=e\wedge\left(  \alpha
u^{+}\right)  $ and $q=e\wedge\left(  \alpha u^{-}\right)  .$ It is clear that
$0\leq p,q\leq e$, $p\wedge q=0$ and $p+q=e\wedge\left(  \alpha u^{+}\right)
+e\wedge\left(  \alpha u^{-}\right)  =e\wedge\left(  \alpha\left\vert
u\right\vert \right)  =e.$ We derive that $p$ and $q$ are components. Observe
now that $u^{+}\leq\beta e=\beta p+\beta q$ for some $\beta\geq\frac{1}%
{\alpha}.$ As $q\wedge u^{+}=0$ we deduce that $u^{+}\leq\beta p.$ Analogously
$u^{-}\leq\gamma q$ for some $\gamma>0.$ Since, by Lemma \ref{1}, $H$ is
positive we obtain $0\leq H\left(  u^{+}\right)  \leq\beta H\left(  p\right)
$ and $0\leq H\left(  u^{-}\right)  \leq\gamma H\left(  q\right)  .$ So $0\leq
H\left(  u^{+}\right)  H\left(  u^{-}\right)  \leq\beta\gamma H\left(
p\right)  H\left(  q\right)  =0,$ where we use Lemma \ref{1} once again. It
follows that $H\left(  u^{+}\right)  H\left(  u^{-}\right)  =0$ and so%
\[%
\begin{array}
[c]{ll}%
H\left(  \left\vert u\right\vert \right)  =H\left(  u^{+}+u^{-}\right)  &
=H\left(  u^{+}\right)  +H\left(  u^{-}\right) \\
& =\left\vert H\left(  u^{+}\right)  -H\left(  u^{-}\right)  \right\vert
=\left\vert H\left(  u^{+}-u^{-}\right)  \right\vert =\left\vert H\left(
u\right)  \right\vert .
\end{array}
\]

$(2)\Rightarrow(3)$ Let $x\in L$ be clean. Then $x=p+u$ where $p\in L$ is a
component and $u\in L$ is a strong unit. It follows that $H\left(  \left\vert
x-p\right\vert \right)  =\left\vert H\left(  x\right)  -H\left(  p\right)
\right\vert .$ By Lemma \ref{1} $H$ is positive and $H\left(  p\right)
\in\left\{  0,1\right\}  .$

If $H\left(  p\right)  =0$ then $H\left(  \left\vert x-p\right\vert \right)
=\left\vert H\left(  x\right)  -H\left(  p\right)  \right\vert =\left\vert
H\left(  x\right)  \right\vert $ and, as $H$ is positive, $H\left(  \left\vert
x\right\vert \right)  =H\left(  \left\vert x\right\vert -p\right)  \leq
H\left(  \left\vert x-p\right\vert \right)  \leq H\left(  \left\vert
x\right\vert +p\right)  =H\left(  \left\vert x\right\vert \right)  ,$from
which we derive that $\left\vert H\left(  x\right)  \right\vert =H\left(
\left\vert x\right\vert \right)  .$

If $H\left(  p\right)  =1$ then $e-x=e-p+\left(  -u\right)  $ with $e-p$ is a
component, $-u$ is a strong unit and $H\left(  e-p\right)  =0.$ So $H\left(
\left\vert e-x\right\vert \right)  =\left\vert H\left(  e-x\right)
\right\vert =\left\vert 1-H\left(  x\right)  \right\vert .$

Assume now that $E$ is a clean vector subspace of $L$ and let $x\in E.$ It
follows that $nx$ is clean for all $n\in\left\{  1,2,...\right\}  $. So
$H\left(  \left\vert nx\right\vert \right)  =\left\vert H\left(  nx\right)
\right\vert $ or $\left\vert H\left(  e-nx\right)  \right\vert =\left\vert
1-H\left(  nx\right)  \right\vert $ for all $n\in\left\{  1,2,...\right\}  .$
If $H\left(  \left\vert nx\right\vert \right)  =\left\vert H\left(  nx\right)
\right\vert $ for some $n\in\left\{  1,2,...\right\}  ,$ then $H\left(
\left\vert x\right\vert \right)  =\left\vert H\left(  x\right)  \right\vert .$
Otherwise, $\left\vert H\left(  e-nx\right)  \right\vert =\left\vert
1-H\left(  nx\right)  \right\vert $ for all $n\in\left\{  1,2,...\right\}  .$
So $H\left(  \left\vert \frac{1}{n}e-x\right\vert \right)  =\left\vert
\frac{1}{n}-H\left(  x\right)  \right\vert $ for all $n\in\left\{
1,2,...\right\}  .$ Now observe that $H\left(  \left\vert \frac{1}%
{n}e-x\right\vert \right)  \underset{n\rightarrow+\infty}{\longrightarrow
}H\left(  \left\vert x\right\vert \right)  $ $\left(  \text{because }H\text{
is positive}\right)  $ and $\left\vert \frac{1}{n}-H\left(  x\right)
\right\vert \underset{n\rightarrow+\infty}{\longrightarrow}\left\vert H\left(
x\right)  \right\vert ,$ from which we derive that $H\left(  \left\vert
x\right\vert \right)  =\left\vert H\left(  x\right)  \right\vert $. This shows
that $H$ acts on $E$ like a lattice homomorphism.

$(3)\Rightarrow(2)$ Let $u\in L$ be a strong unit and observe that $\left\{
\lambda u:\lambda\in\mathbb{R}\right\}  $ is a clean vector subspace of $L$.
So $H\left(  \left\vert \lambda u\right\vert \right)  =\left\vert H\left(
\lambda u\right)  \right\vert $ for all $\lambda\in\mathbb{R}$ and then
$H\left(  \left\vert u\right\vert \right)  =\left\vert H\left(  u\right)
\right\vert .$

$(2)\Rightarrow(1)$ Observe first that, by Lemma \ref{1}, $H$ is positive. Now
let $u\in L$ be a strong unit of $L$ and $\alpha>0$ be such that $e\leq
\alpha\left\vert u\right\vert .$ As $H$ is positive, $1=H\left(  e\right)
\leq\alpha H\left(  \left\vert u\right\vert \right)  =\alpha\left\vert
H\left(  u\right)  \right\vert .$ So $H\left(  u\right)  \neq0$ and we are done.
\end{proof}

\begin{corollary}
Assume that $L$ is clean and let $H:L\rightarrow \mathbb{R}$ be a unital
linear functional. The following are equivalent.
\begin{enumerate}
\item $H$ preserves units.
\item $H$ acts like a lattice homomorphism on the set of all strong units.
\item $H$ is a lattice homomorphism.
\end{enumerate}
\end{corollary}

In the litterature, clean vector lattices with a strong unit $e>0$ are
characterized in many ways \cite{BouSmiti,Hager}. Our last result in this work
is a contribution to this topic. First, recall that the space $Max\left(
L\right)  $ of all maximal ideals in the vector lattice $L$ with a strong unit
$e>0$, equipped with the hull-kernel topology is a compact Hausdorff space.
Also, for all $x\in L$ and all $J\in Max\left(  L\right)  $ there exists a
unique real denoted by $\symbol{94}x\left(  J\right)  $ such that
$x-\symbol{94}x\left(  J\right)  e\in J.$ For all $x\in L,$ the function
$\symbol{94}x$ is a real-valued continuous function on $Max\left(  L\right)
$. Moreover, the set $\symbol{94}L=\left\{  \symbol{94}x:x\in L\right\}  $
separates the points of $Max\left(  L\right)  $. For more details on this
subject the reader can consult the standard monograph \cite{RieszI}.

\begin{theorem}
The following are equivalent.

\begin{enumerate}
\item Every unit preserving unital linear functional $H:L\rightarrow
\mathbb{R}$ is a lattice homomorphism.

\item $L$ is clean.
\end{enumerate}
\end{theorem}

\begin{proof}
We only prove that $(1)\Rightarrow(2).$ Assume, by contrapositive, that $L$ is
not clean. By \cite[Theorem 3.5]{Hager}, $Max\left(  L\right)  $ is not
totally disconnected. So one of its connected components, say $C,$ contains
$2$ different elements $J_{1}$ and $J_{2}.$ Let $H:L\rightarrow\mathbb{R}$ be
the linear functional defined by $H\left(  x\right)  =\frac{1}{2}\left(
\symbol{94}x\left(  J_{1}\right)  +\symbol{94}x\left(  J_{2}\right)  \right)
.$ As $C$ is connected $\symbol{94}x\left(  C\right)  $ is an interval in
$\mathbb{R}$ for all $x\in L.$ So $H\left(  x\right)  =\frac{1}{2}\left(
\symbol{94}x\left(  J_{1}\right)  +\symbol{94}x\left(  J_{2}\right)  \right)
\in\symbol{94}x\left(  C\right)  $ for all $x\in L.$ If $x\in L$ is a strong
unit then $0\notin\symbol{94}x\left(  C\right)  ,$so $H\left(  x\right)
\neq0.$ Since $\symbol{94}L=\left\{  \symbol{94}x:x\in L\right\}  $ separates
the points of $Max\left(  L\right)  $ we can find $x\in L$ such that
$\symbol{94}x\left(  J_{1}\right)  =1$ and $\symbol{94}x\left(  J_{2}\right)
=-1.$ So $0=\left\vert H\left(  x\right)  \right\vert \neq H\left(  \left\vert
x\right\vert \right)  =1$ and then $H$ is not a lattice homomorphism.
\end{proof}

\end{document}